\documentclass[10pt]{amsart}

\usepackage{graphicx}
\usepackage{hyperref}

\newtheorem{theorem}{Theorem}[section]
\newtheorem{lemma}[theorem]{Lemma}

\theoremstyle{definition}
\newtheorem{definition}[theorem]{Definition}

\newtheorem{formula}{Formula}

\theoremstyle{remark}
\newtheorem{remark}[theorem]{Remark}

\numberwithin{equation}{section}

\begin{document}

\title[Classifying solutions of ${\rm SU}(n+1)$ Toda system]{Classifying solutions of ${\rm SU}(n+1)$ Toda system around a singular source}

\author{Jingyu Mu}
\address{School of Mathematical  Sciences, University of Science and Technology of China, Hefei 230026 China}
\email{jingyu@mail.ustc.edu.cn}
\thanks{
Y.S. is supported in part by NSFC (Grant No. 11931009). 
B.X. is supported in part by the Project of Stable Support for Youth Team in Basic Research Field, CAS (Grant No. YSBR-001) and 
NSFC (Grant Nos. 12271495, 11971450 and 12071449).  
}
\thanks{$^\dagger$B.X. is the corresponding author.}

\author{Yiqian Shi}
\address{CAS Wu Wen-Tsun Key Laboratory of Mathematics and School of Mathematical  \newline \indent Sciences, University of Science and Technology of China, Hefei 230026 China}
\email{yqshi@ustc.edu.cn}
\thanks{}

\author{Tianyang Sun}
\address{School of Mathematical  Sciences, University of Science and Technology of China, Hefei 230026 China}
%\curraddr{Department of Mathematics, UC Berkeley, Berkeley,  CA 94720 USA}
\email{tysun@mail.ustc.edu.cn}

%\author{Jijian Song}
%\address{Center for Applied Mathematics, School of Mathematics, Tianjin University\newline \indent Tianjin, 300350, China}
%\email{smath@tju.edu.cn}
%\thanks{}

\author{Bin Xu$^\dagger$}
\address{CAS Wu Wen-Tsun Key Laboratory of Mathematics and  School of Mathematical \newline \indent Sciences, University of Science and Technology of China, Hefei 230026 China}
\email{bxu@ustc.edu.cn}
%    General info

\date{}

\dedicatory{}


\begin{abstract}
Consider a positive integer $n$ and $\gamma_1>-1,\cdots,\gamma_n>-1$. Let $D=\{z\in {\Bbb C}:|z|<1\}$, and let $(a_{ij})_{n\times n}$ denote the Cartan matrix of $\frak{su}(n+1)$. Utilizing the ordinary differential equation of $(n+1)$th order around a singular source of ${\rm SU}(n+1)$ Toda system, as discovered by Lin-Wei-Ye ({\it Invent Math}, {\bf 190}(1):169-207, 2012), we precisely characterize a solution $(u_1,\cdots, u_n)$ to the ${\rm SU}(n+1)$ Toda system
\begin{equation*}
		\begin{cases} \frac{\partial^2 u_i}{\partial z\partial \bar z}+\sum_{j=1}^n a_{ij} e^{u_j}&=\pi \gamma _i\delta _0\,\,{\rm on}\,\, D\\
			\frac{\sqrt{-1}}{2}\,\int_{D\backslash \{0\}} e^{u_{i} }{\rm d}z\wedge {\rm d}\bar z &<   \infty
		\end{cases}  \quad \text{for all}\quad i=1,\cdots, n \end{equation*}
using $(n+1)$ holomorphic functions that satisfy the normalized condition. Additionally, we demonstrate that for each $1\leq i\leq n$, $0$ represents the cone singularity with angle $2\pi(1+\gamma_i)$ for the metric $e^{u_i}|{\rm d}z|^2$ on $D\backslash\{0\}$, which can be locally characterized by $(n-1)$ non-vanishing holomorphic functions at $0$.
\end{abstract}

\maketitle

%%%%%%%%%%%%%%%%%%%%%%%%%%%%%%%%%%%%%%%%%%%%%%%%%%%%%%%%%%%%%%%%%%%%%%%%
\section{Introduction}

%Physicists (\cite[Doliwa:1997]) already knew in the 1990s that local solutions to open Toda systems associated to non-exceptional simple Lie algebras are equivalent to the holomorphic curves in complex projective spaces satisfying extra properties. 
Gervais-Matsuo \cite[Section 2.2.]{GM:1993} firstly showed that totally un-ramified holomorphic curves in ${\Bbb P}^n$ induce local solutions to ${\rm SU}(n+1)$ Toda systems in the sense that these systems are actually the infinitesimal Pl\" ucker formulae for these curves. 
A. Doliwa \cite{Do:1997} generalized their result to Toda systems associated with non-exceptional simple Lie algebras.  There have been lots of research works on the classification of solutions of Toda systems of various types which satisfy some boundary conditions on the punctured Riemann surfaces since then. We list some relevant results as follows.

Jost-Wang \cite[Theorem 1.1.]{JW:2002} classified all solutions $(u_1,\cdots, u_n)$ to the
${\rm SU}(n+1)$ Toda system on ${\Bbb C}$ satisfying the so-called finite energy condition:
\begin{equation*}
		\begin{cases} \frac{\partial^2 u_i}{\partial z\partial \bar z}+\sum_{j=1}^n a_{ij} e^{u_j}&=0\,\,{\rm on}\,\, {\Bbb C}\\
			\frac{\sqrt{-1}}{2}\,\int_{{\Bbb C}} e^{u_{i} }{\rm d}z\wedge {\rm d}\bar z &<   \infty
		\end{cases} \quad \text{for all}\quad i=1,\cdots, n.\end{equation*} 
Here  we recall
$(a_{ij})=\begin{pmatrix}
			2 & -1 & 0 & \cdots  & \cdots  & 0\\
			-1 & 2 & -1 & 0 & \cdots & 0\\
			0 & -1 & 2 & -1 & \cdots & 0\\
			\vdots & \vdots & \vdots & \vdots & \vdots & \vdots\\
			0 & \cdots & \cdots & -1 & 2 & -1\\
			0 & \cdots & \cdots & 0 & -1 & 2
		\end{pmatrix}$.
Equivalently, they proved that any holomorphic curve ${\Bbb C}\to {\Bbb P}^n$ 
associated with such a solution can be compactified to a rational normal curve ${\Bbb P}^1\to {\Bbb P}^n$ (\cite[Theorem 1.2]{JW:2002}). 
Consequently, the space of such solutions is isomorphic to 
${\rm PSL}(n+1,\,{\Bbb C})/{\rm PSU}(n+1)$ and has dimension $n(n+2)$. By using the value distribution theory of holomorphic curves, A. Eremenko \cite[Theorem 2]{Ere:2007} made the classification for a larger class of solutions $u=(u_1,\cdots, u_n)$ of the ${\rm SU}(n+1)$ Toda system on ${\Bbb C}$ than  \cite[Theorem 1.1.]{JW:2002} under the condition that as $R\to \infty$ there holds
\[\frac{\sqrt{-1}}{2}\,\int_{|z|<R} e^{u_{1} }{\rm d}z\wedge {\rm d}\bar z=O(R^K)
\quad {\rm for\,\, some}\quad K\geq 0.\]
Jost-Lin-Wang \cite[Proposition 3.1.]{JLW:2006} described for the first time the asymptotic behavior of solutions of the ${\rm SU}(n+1)$
Toda system on ${\Bbb C}\backslash \{0\}$ near singular source $0$. 
For the ${\rm SU}(n+1)$ Toda system on the twice-punctured Riemann sphere with finite energy:
\begin{equation*}
		\begin{cases} \frac{\partial^2 u_i}{\partial z\partial \bar z}+\sum_{j=1}^n a_{ij} e^{u_j}&=\pi \gamma _i\delta _0\,\,{\rm on}\,\, {\Bbb C}\quad (\gamma_i>-1)\\
			\frac{\sqrt{-1}}{2}\,\int_{{\Bbb C}\backslash \{0\}} e^{u_{i} }{\rm d}z\wedge {\rm d}\bar z &<   \infty
		\end{cases}  \quad \text{for all}\quad i=1,\cdots, n. \end{equation*}
Lin-Wei-Ye \cite[Theorem 1.1.]{LWY:2012} classified all its solutions, by which they generalized the result of Jost-Wang. 
The space of these solutions has dimension at most $n(n+2)$. 
Karmakar-Lin-Nie-Wei \cite[Theorem 1.1.]{KLNW2022} obtained the classification of all solutions to the elliptic Toda system associated with a general simple Lie algebra, where the space of all solutions is also of finite dimension.  
Chen-Lin \cite{CL:2022} classified all even solutions to some ${\rm SU}(3)$ Toda systems with critical parameters on tori.

The ${\rm SU}(2)$ Toda system coincides with the  Liouville equation
$\frac{\partial^2 u_1}{\partial z\partial \bar z}+2e^{u_1}=0$, 
whose local solutions $u_1$ are induced by non-degenerate meromorphic functions (\cite{liouville1853equation}) 
and define metrics $e^{u_1}|{\rm d}z|^2$ with Gaussian curvature $4$. 
R. Bryant \cite[Proposition 4]{Bryant:1987} show that if such a metric  
$e^{u_1}|{\rm d}z|^2$ on the punctured disk $D^*:=\{z\in {\Bbb C}:0<|z|<1\}$ has finite area, i.e.  
$\frac{\sqrt{-1}}{2}\,\int_{D^*}\, e^{u_{1} }{\rm d}z\wedge {\rm d}\bar z <\infty$, then near $0$,  
$e^{u_1}|{\rm d}z|^2$ could be expressed by  
$\frac{(\gamma_1+1)^2|\xi|^{2\gamma_1}|{\rm d}\xi|^2}{(1+|\xi|^{2\gamma_1+2})^2}$
for some  constant $\gamma_1>-1$, under another complex coordinate $\xi=\xi(z)$ which is defined near $0$ and preserves $0$, i.e. $\xi(0)=0$.
Moreover, Chou-Wang \cite[Corollary 2]{Chou_Wan:1994} provided a classification of all solutions with finite energy for the Liouville equation over $\mathbb{C}\backslash\{0\}$. Similarly, Prajapat-Tarantello \cite[Theorem 1.1.]{PT2001} accomplished a classification comparable to Chou-Wang's but for a more generalized equation over $\mathbb{C}$. 
We briefly address the discrepancy between the classification result obtained by Chou-Wan and Prajapat-Tarantello, and the one by Bryant. The former yields solutions with three real parameters, whereas the latter involves only one parameter. Bryant accomplished this by using a wealth of complex coordinate changes near $0$, while preserving $0$ to simplify solutions in the latter classification. In particular, Bryant's classification is not sensitive at all to whether or not $\gamma_1$ is an integer.

%Consequently, the former involves three real parameters, whereas the latter involves only one.

%In other words, under this coordinate transformation $z\mapsto w(z)$ near $0$, the solution $u_1$ to 
%the old Liouville equation $\frac{\partial^2 u_1}{\partial z\partial \bar z}+2e^{u_1}=0$
%in $D^*=\{0<|z|<1\}$ is replaced by  
%\[v_1:=2\gamma_1\log\,|\xi|+2\log\,\frac{\gamma_1+1}{1+|\xi|^{2\gamma+2}}\] to 
%the new one $\frac{\partial^2 v_1}{\partial \xi\partial \bar \xi}+2e^{v_1}=0$ in a small neighborhood of $0$. 

By using the ordinary differential equation of $(n+1)$th order around a singular source of
${\rm SU}(n+1)$ Toda system discovered by Lin-Wei-Ye \cite[p.201, (7.1)]{LWY:2012}, 
we generalize in Theorem \ref{thm:model} (ii) the result of R. Bryant by classifying 
all solutions $u=(u_1,\cdots, u_n)$ to the following ${\rm SU}(n+1)$ Toda system
\begin{equation}
\label{equ:Toda}
		\begin{cases} \frac{\partial^2 u_i}{\partial z\partial \bar z}+\sum_{j=1}^n a_{ij} e^{u_j}&=\pi \gamma _i\delta _0\,\,{\rm on}\,\, D \quad (\gamma_i>-1) \\
			\frac{\sqrt{-1}}{2}\,\int_{D^*} e^{u_{i} }{\rm d}z\wedge {\rm d}\bar z &<   \infty
		\end{cases}  \quad \text{for all}\quad i=1,\cdots, n. \end{equation}
Roughly speaking, we establish a correspondence between solutions $u=(u_1,\cdots, u_n)$ to \eqref{equ:Toda} and $(n+1)$ holomorphic functions satisfying the normalized condition on $D$. 
Moreover, for each $1\leq i\leq n$, we could characterize the germs at $0$ of metric 
$e^{u_j}|{\rm d}z|^2$ with cone angle
$2\pi(1+\gamma_i)$ at $0$ in terms of some $(n-1)$ holomorphic functions non-vanishing at $0$.
Before the statement of Theorem \ref{thm:model}, we prepare some notations. Recall that the inverse matrix $(a^{ij})_{n\times n}$ of $(a_{ij})_{n\times n}$ satisfies
$a^{ij}=\frac{j(n+1-i)}{n+1}$ for all $1\leq j\leq i\leq n$.
Define $\alpha_i:=\sum_{j=1}^n\, a^{ij}\gamma_j$
for $i=1,\cdots, n$, and set 
\begin{equation}
\label{equ:beta}
\begin{cases}
\beta_0&:=-\alpha_1,\\
\beta_i&:=\alpha_i-\alpha_{i+1}+i\quad \text{for}\quad 1\leq i\leq n-1,\\
\beta_n&:=\alpha_n+n\,.
\end{cases}
\end{equation}
Then, by the very definition of $\beta_i$'s, we have $\beta_i-\beta_{i-1}=\gamma_i+1>0$ for all $i=1,\cdots, n$,
$\beta_0<\beta_1<\cdots<\beta_n$, and $\beta_0+\beta_1+\cdots+\beta_n=n(n+1)/2$.
For any $(n+1)$ holomorphic functions $g_0(z),\cdots, g_k(z)$ on $D$ with $0\leq k\leq n$, we define

 {\small 
\begin{equation}
\label{equ:G}
 G_k\big(\beta_0,\cdots,\beta_k;g_0(z),\cdots, g_k(z);z\big):= 
 z^{k(k+1)/2-(\beta_0+\cdots+\beta_k)}\cdot W\left(z^{\beta _0}g_0(z), \cdots , z^{\beta_k}g_k(z)\right),
\end{equation} 
}
where $W\left(z^{\beta _0}g_0(z),\, z^{\beta _1}g_1(z),\, \cdots,\, z^{\beta_k}g_k(z)\right)$ equals 
\[\begin{vmatrix}
				z^{\beta _0}g_0(z)   & z^{\beta _1}g_1(z) & \cdots  & z^{\beta_k}g_k(z) \\
				\big(z^{\beta_0}g_0(z)\big)' & \big(z^{\beta_1}g_1(z)\big)' & \cdots  & \big(z^{\beta_n}g_k(z)\big)' \\
				\vdots  & \vdots  & \cdots  & \vdots \\
				\big(z^{\beta_0}g_0(z)\big)^{(k)}  & \big(z^{\beta_1}g_1(z)\big)^{(k)} & \cdots  & \big(z^{\beta_k}g_k(z)\big)^{(k)}
			\end{vmatrix}.\]
Then $G_k$ is holomorphic on $D$ and satisfies 
\begin{equation}
\label{equ:G_k}
G_k|_{z=0}=\Pi_{i=0}^k\, g_i(0)\cdot\Pi_{0\leq i<j\leq k}\, (\beta_i-\beta_j)\end{equation}
by Lemma \ref{lem:normalize}. In particular, 
$G_n\big(\beta_0,\cdots,\beta_n;g_0(z),\cdots, g_n(z);z\big)$ coincides with \\
$W\left(z^{\beta _0}g_0(z),\, z^{\beta _1}g_1(z),\,\cdots, z^{\beta_n}g_n(z)\right)$ since $\sum_{i=0}^n\,\beta_i=\frac{n(n+1)}{2}$.  

\begin{definition}
\label{equ:normalized}
We call that the $(n+1)$ holomorphic functions $g_0(z),\cdots , g_n(z)$ on $D$ satisfy the 
{\it normalized condition} if and only if 
$G_n\big(\beta_0,\cdots,\beta_n;g_0(z),\cdots, g_n(z);z\big)\equiv 1$ on 
$D$. 
In particular, $g_0,\cdots, g_n$ do not vanish at $0$ by \eqref{equ:G_k}.
\end{definition}
%\newpage

\begin{theorem}
\label{thm:model}
Let $u=(u_1,\cdots, u_n)$ be a solution to the ${\rm SU}(n+1)$ Toda system \eqref{equ:Toda}. Then we have the following two statements.
\begin{enumerate}
\item[(i)] There exist $(n+1)$ holomorphic functions $g_0,\cdots, g_n$ satisfying the normalized condition on $D$ such that for each $1\leq k\leq n$, 
\begin{equation}
\label{equ:curve_sol}
u_k=-\sum_{j=1}^n a_{kj}\log\| \Lambda_{j-1}(\nu)\|^2, \end{equation}
where $[\nu]=[\nu_0,\cdots,\nu_n]:D^*\to {\Bbb P}^n$ is the multi-valued holomorphic curve defined by
\begin{equation}
\label{equ:can}
z\mapsto \left[z^{\beta_0}g_0(z),\,z^{\beta_1}g_1(z),\,\cdots,\, z^{\beta_n}g_n(z)\right]\end{equation}
and the definition of $\Lambda_{i}(\cdot)$ will be given in Section 2.
In particular, $u_k$ equals $2\gamma_k\,\log\,|z|$ plus a H\" older continuous remainder $R_k$ near $0$, where
{\small 
\begin{eqnarray*}
R_k&=&-\sum\limits_{j=1}^n a_{kj}\log\, r_j\quad {\rm with}\\
r_j&=&\left| G_{j-1} 
\big(\beta_0, \beta_1,\cdots ,\beta_{j-1} ;
	g_0(z),g_1(z),\cdots ,g_{j-1}(z);z\big) \right|^{2}+\\
	&\sum\limits_{0 \le i_0 < i_1< \cdots < i_{j-1}\le n \atop i_{j-1}> j-1 }&
	\left| z \right| ^{2\left(\sum\limits_{l=0}^{j-1} \beta _{i_l}-\frac{(j-1)j}{2}+\alpha_j\right)}
	\left | G_{j-1}\big(\beta _{i_0},\beta _{i_1},\cdots,\beta_{i_{j-1}};g_{i_0}(z),g_{i_1}(z),\cdots,g_{i_{j-1}}(z);z\big)\right|^2.  
\end{eqnarray*}
}
Moreover, any curve in the form of \eqref{equ:can} can yield a solution $u=(u_1,\cdots, u_n)$ to \eqref{equ:Toda} through \eqref{equ:curve_sol}, even if the integral condition in \eqref{equ:Toda} is relaxed to $\frac{\sqrt{-1}}{2}\,\int_{0<|z|<r} e^{u_{i} }{\rm d}z\wedge {\rm d}\bar z < \infty$ for all $0<r<1$.

\item[(ii)] For all $1\leq k\leq n$, metrics $e^{u_k}|dz|^2$ have cone angle $2\pi(1+\gamma_k)$ at $z=0$. And there exist a complex coordinate change $z\mapsto \xi=\xi(z)$ near $z=0$ and preserving $0$, and $(n-1)$ holomorphic functions $\tilde g_2(\xi),\cdots, \tilde g_n(\xi)$ non-vanishing at $0$ such that these $n$ metrics 
near $0$ could be expressed in terms of these $(n-1)$ functions and $\{\beta_i\}_{i=0}^n$. In particular, 
$e^{u_1}|{\rm d}z|^2$ near $z=0$ could be simplified into the form of
{\small 
\begin{equation*}
 \left | \xi \right | ^{2\gamma_1}\frac{(\beta_1-\beta_0)^2
		+\sum\limits_{0\le i_0< i_1\le n \atop i_1> 1} \left | \xi \right |^{2(\beta_{i_0}+\beta_{i_1}-1+\alpha_2)}
		\left | G_1\big(\beta_{i_0},\beta_{i_1};\tilde{g}_{i_0}(\xi),\tilde{g}_{i_1}(\xi);\xi \big) \right |^2 }
	{\Big(1+\left | \xi \right |^{2(\beta_1-\beta_0)}+\left | \xi \right |^{2(\beta_2-\beta_0)}\left |\tilde{g}_2(\xi) \right |^2
		+\left | \xi \right |^{2(\beta_n-\beta_0)}\left |\tilde{g}_n(\xi) \right |^2 \Big)^2 }
	|{\rm d}\xi|^2.    
\end{equation*}
}
 
\end{enumerate}
\end{theorem}

\begin{remark}
Theorem 2.1 (i) refines the asymptotic estimate around a singular source of solutions to ${\rm SU}(n+1)$ Toda system  in \cite[Lemma 2.1]{JW:2002} and \cite[Theorem 1.3 (i)]{LWY:2012} to the effect that it gives the bounded remainders of $u_k$'s explicitly, which are actually H\" older continuous at $0$ and smooth outside $0$.
%By simple computation, metric $e^{u_1}|{\rm d}z|^2$ in Case $n=1$ of Statement (ii) coincides with $\frac{(\gamma_1+1)^2|\xi|^{2\gamma_1}|{\rm d}\xi|^2}{(1+|\xi|^{2\gamma_1+2})^2}$ already given by R. Bryant \cite[Proposition 4]{Bryant:1987}. 
\end{remark}

\begin{remark}
\label{rem:global_analytic}
In this note, we utilize $z^\beta$ and $z^\beta \log\, z$, two multi-valued analytic functions with $\beta\in {\Bbb R}$ on $D^*$.  Following Ahlfors (\cite[Section 8.1.]{Ahlfors:1978}), they fall under the category of \textit{global analytic functions}, having analytic germs at each point in $D^*$. 
In particular, the values derived from germs of $\sqrt{z}\log\, z$ at $z=\frac{1}{2}$ form a countable unbounded subset 
$
\left\{(-1)^m\frac{\sqrt{2}}{2}\big(-\ln\, 2+2\pi m\sqrt{-1}\big): m\in \mathbb{Z}\right\}
$ of ${\Bbb C}$. 
\end{remark}

We conclude the introduction by elucidating the structure of the subsequent three sections of this manuscript. In Section 2, considering a not-necessarily simply connected domain $\Omega\subset {\Bbb C}$, we establish a correspondence between solutions to the ${\rm SU}(n+1)$ Toda system on $\Omega$ and totally unramified unitary curves $\Omega\to {\Bbb P}^n$ (see Definition \ref{defn:unitary} and Lemma \ref{lem:Toda_map}). This correspondence is such that the solutions are induced by the infinitesimal Pl\"ucker formulae of the curves, a generalization of the simply connected case employed by Jost-Wang \cite[Section 3]{JW:2002}, based on \cite[Section 2.4]{GH:1994}. 
In Section 3, utilizing the ordinary differential equation of $(n+1)$th order around $z=0$ as discovered by Lin-Wei-Ye \cite{LWY:2012}, we establish the first part of Statement (i) of Theorem \ref{thm:model}. This part asserts that a solution $u$ to \eqref{equ:Toda} is induced by the infinitesimal Pl\"ucker formulae of the canonical unitary curve 
$z\mapsto \left[z^{\beta_0}g_0(z),\,z^{\beta_1}g_1(z),\,\cdots,\, z^{\beta_n}g_n(z)\right]$ on $D^*$,
where $g_0,\cdots, g_n$ are $(n+1)$ holomorphic functions satisfying the normalized condition on $D$. 
The last section is dedicated to proving the remaining part of Theorem \ref{thm:model} by applying the infinitesimal Pl\"ucker formulae to this canonical unitary curve.

\section{Correspondence between curves and solutions}
Jost-Wang \cite[Section 3]{JW:2002} established a correspondence between
solutions to ${\rm SU}(n+1)$ Toda system on a simply connected domain in ${\Bbb C}$ 
and totally-unramified holomorphic curves from this domain to ${\Bbb P}^n$.
In this section, we generalize their correspondence to a not-necessarily simply connected domain 
$\Omega\subset {\Bbb C}$. Before the statement of the more general correspondence, we prepare some notations as follows, where we 
use Griffiths-Harris \cite[Section 2.4]{GH:1994} as a general reference.

\begin{definition}
\label{defn:unitary}
We generalize the concept of associated  curves in \cite[pp.263-264]{GH:1994} to the multi-valued case in the following: 
\begin{enumerate}
    \item We call $f:\Omega\to {\Bbb P}^n$ a {\it projective holomorphic curve} if and only if it satisfies the following three
conditions:
\begin{enumerate}
\item[(i)] $f$ is a multi-valued holomorphic map; 

\item[(ii)] $f$ is non-degenerate, i.e. the image of a germ ${\frak f}_z$ of $f$ at any point $z\in \Omega$ is not contained in a hyperplane of ${\Bbb P}^n$; and

\item[(iii)] the monodromy representation of $f$ is a group homomorphism 
${\mathcal M}_f:\pi_1(\Omega,\, B)\to {\rm PSL}(n+1,\,{\Bbb C})$,
where ${\rm PSL}(n+1,\, {\Bbb C})$ is the holomorphic automorphism group of ${\Bbb P}^n$
(\cite[pp.64-65]{GH:1994}) and $B\in\Omega$ is a base point. We also say that $f$ {\it has monodromy
in} ${\rm PSL}(n+1,\, {\Bbb C})$  briefly. 
\end{enumerate}

\item We call such a curve $f$ {\it unitary} if and only if it has monodromy in 
${\rm PSU}(n+1)$, which is the group of rigid motions with respect to the Fubini-Study metric
$\omega_{\rm FS}=\frac{\sqrt{-1}}{2\pi}\,\partial\overline{\partial}\,\log\, \|Z\|^2$
with $Z\in {\Bbb C}^{n+1}-\{0\}$ 
on ${\Bbb P}^n={\Bbb P}({\Bbb C}^{n+1})$ (\cite[pp.30-31]{GH:1994}). 
Mimicking the definition in \cite[pp.263-264]{GH:1994}, for a unitary curve 
$f:\Omega\to {\Bbb P}^n$, we could define its {\it $k$th associated curve}
\[f_k:\Omega\to G(k+1,\, n+1)\subset {\Bbb P}\big(\Lambda^{k+1}{\Bbb C}^{n+1}\big)\quad {\rm for\,\, all}\quad k=0,1,\cdots, (n-1),\]
which are also unitary curves.

\item We call a unitary curve $f:\Omega\to {\Bbb P}^n$ {\it totally un-ramified} if and only if
for each point $z\in \Omega$, each germ $\frak f$ of $f$ is totally un-ramified, i.e. there exists
a lifting $\widehat{{\frak f}}:U_z\to {\Bbb C}^{n+1}$ of $\frak f$ such that its $n$th associated curve 
\[ \widehat{{\frak f}}\wedge \widehat{{\frak f}}'(z)\wedge\cdots\wedge 
\widehat{{\frak f}}^{(n)}(z): U_z \to \Lambda^{n+1}\big({\Bbb C}^{n+1}\big) \] 
equals $e_0\wedge e_1\wedge\cdots\wedge e_n$ identically, where $U_z\subset \Omega$ is some open neighborhood of $z$ and $\{e_0,\cdots, e_n\}$ is the standard ortho-normal basis of ${\Bbb C}^{n+1}$. Hence, the $n$th associated curve
$f_n$ of $f$ is also well defined. Note that a totally un-ramified curve must be non-degenerate.

\end{enumerate} 
\end{definition}

We observe that the infinitesimal Pl\" uck formulae \cite[p.269]{GH:1994} also hold for unitary curves beside
single-valued holomorphic curves and they induce solutions to ${\rm SU}(n+1)$ Toda system 
in the following:

\begin{lemma}
\label{lem:curve_Omega}
Let $f:\Omega\to {\Bbb P}^n$ be a unitary curve and $f_0:=f,\, f_1,\cdots, f_{n-1}$ its associated curves. Let $\frak f$ be a germ of $f$ and $\hat{\frak f}$ be one of its lifting. 
Then $\Lambda_k(\frak f,\, z)=
\hat{\frak f}(z)\wedge\hat{\frak f}'(z)  \cdots \wedge \hat{\frak f}^{(k)}(z)\in \Lambda^{k+1}{\Bbb C}^{n+1}$ is a lifting of
some germ ${\frak f}_k$ of $f_k$. Endow $\Lambda^{k+1}\big({\Bbb C}^{n+1}\big)$'s  
with induced metrics from $\big({\Bbb C}^{n+1},\,\|\cdot\|\big)$ for $k=0,1,\cdots, n$, and
set $\|\Lambda_{-1}\|\equiv 1$. 
\begin{enumerate}
\item[(i)] {\rm (Infinitesimal Pl\" uck formula)} For $k=0,1,\cdots,(n-1)$, we have 
\begin{equation}
\label{equ:Plucker}
f^*_k\omega_{\rm FS}=\frac{\sqrt{-1}}{2\pi}\,\frac{\|\Lambda_{k-1}(f)\|^2\cdot\|\Lambda_{k+1}(f)\|^2 }{\|\Lambda_k(f)\|^4}\, {\rm d}z\wedge {\rm d}\bar z,
\end{equation}
where we write the notion of $\Lambda_{\cdot}(f)$ on purpose since 
$\frac{\|\Lambda_{k-1}(f)\|^2\cdot\|\Lambda_{k+1}(f)\|^2 }{\|\Lambda_k(f)\|^4}$
on the right-hand side does not depend on the choice of the lifting $\hat{\frak f}$ of $f$.

\item[(ii)] {\rm (From curves to solutions)}  Assume furthermore that the unitary curve $f:\Omega\to{\Bbb P}^n$ is totally un-ramified. Then we could choose the lifting $\hat{\frak f}$ of germ $\frak f$ of $f$ in {\rm (i)} such that 
$$\Lambda_n(\frak f,\,z)=\hat{\frak f}(z)\wedge\hat{\frak f}'(z)  \cdots \wedge \hat{\frak f}^{(n)}(z)\equiv 
e_0\wedge\cdots\wedge e_n\in \Lambda^{n+1}{\Bbb C}^{n+1}\quad {\rm on}\quad \Omega.$$
In particular, $\|\Lambda_n\|\equiv 1$.
Then it induces a solution
$u=(u_1,\cdots, u_n)$ to the ${\rm SU}(n+1)$ Toda system 
\begin{equation}
\label{equ:Toda_Omega}
 \frac{\partial^2 u_i}{\partial z\partial \bar z}+\sum_{j=1}^n a_{ij} e^{u_j}=0\quad {\rm on}\quad \Omega\quad 
 {\rm for\,\, all}\quad i=1,\cdots, n. 
\end{equation}
in such a way that
\begin{equation}
\label{equ:sol_Omega}
u_i:=-\sum_{j=1}^na_{ij}\log\| \Lambda_{j-1}(f)\|^2
=\begin{cases}
\log\,\frac{\|\Lambda_{1}(f)\|^2}{\|\Lambda_{0}(f)\|^4}\quad &{\rm for}\quad i=1, \\
\log\,\frac{\|\Lambda_{i-2}(f)\|^2\cdot\|\Lambda_{i}(f)\|^2 }{\|\Lambda_{i-1}(f)\|^4}
\quad &{\rm for\,\, all}\quad i=2,\,3,\,\cdots\,, n-1,\\
\log\,\frac{\|\Lambda_{n-2}(f)\|^2}{\|\Lambda_{n-1}(f)\|^4}\quad &{\rm for}\quad i=n.
\end{cases} 
\end{equation}

\end{enumerate}
\end{lemma} 
\begin{proof} Since $f$ and all its associated curves are unitary, the norm of 
$\Lambda_k(\frak f,\, z)=v(z)\wedge \cdots \wedge v^{(k)}(z)\in \Lambda^{k+1}{\Bbb C}^{n+1}$ 
does not depend on the choice of germ $\frak f$. Hence the infinitesimal Pl\" ucker formulae \eqref{equ:Plucker} follows from the same argument as in \cite[pp.269-270]{GH:1994}. Statement (ii) follows from these formulae and the same argument as in \cite[Section 3.4]{JW:2002}.
\end{proof}

Jost-Wang \cite[Section 2.1]{JW:2002} introduced the Toda map associated with a solution
to the ${\rm SU}(n+1)$ Toda system on a simply connected domain in ${\Bbb C}$.
To obtain our correspondence, we need to introduce the notion of multi-valued
Toda map on $\Omega\subset {\Bbb C}$. 
Let $u=(u_1,\cdots,u_n)$ be an $n$-tuple of real-valued smooth function on $\Omega$ and
the $(n+1)$-tuple $w=(w_0,\cdots, w_n)$ of functions on $\Omega$ be defined by
\begin{equation}
\begin{cases}
	w_0:=-\frac{\sum_{i=1}^n(n-i+1)u_i}{2(n+1)}\\
	w_i:=w_0+\frac{1}{2}\sum_{j=1}^iu_j,\quad 1\leq i\leq n.\label{1.3}
\end{cases}
\end{equation}
Then $u=(u_1,\cdots,u_n)$ solves the
SU(n+1) Toda system \eqref{equ:Toda_Omega} if and only if $w$ satisfies the Maurer-Cartan equation 
$\mathcal{U}_z-\mathcal{V}_{\bar{z}}=[\mathcal{U},\mathcal{V}]$,
where
\begin{equation}
\mathcal{U}=
\begin{pmatrix}
(w_0)_z&&&\\
&(w_1)_z&&\\
&&\ddots&\\
&&&(w_n)_z
\end{pmatrix}
+\begin{pmatrix}
0&&&\\
e^{w_1-w_0}&0&&\\
&\ddots&\ddots&\\
&&e^{w_n-w_{n-1}}&0
\end{pmatrix}
\nonumber
\end{equation}
and $\mathcal{V}=-{\mathcal U}^*=-\overline{{\mathcal U}}^{\rm T}$.
By using the Frobenius theorem and the analytic-continuation-like argument (See \cite[Section 3.1]{JW:2002} and \cite[Chapter 3]{sharpe2000differential}),
we obtain a set of {\it multi-valued} Toda maps $\phi:\Omega\to {\rm SU}(n+1)$ associated with solution $u$ of \eqref{equ:Toda_Omega} such that 
\begin{equation}
\label{equ:Toda_map}
\phi^{-1}{\rm d}\phi=\mathcal{U}{\rm d}z+\mathcal{V}{\rm d}\bar{z}
\end{equation}
and the monodromy of $\phi$ is a group homomorphism
${\mathcal M}_\phi:\pi_1(\Omega,\, B)\to {\rm SU}(n+1)$. Moreover, any two such Toda maps have the difference
of a constant multiple in ${\rm SU}(n+1)$ from the left-hand side, and the set of all the Toda maps associated with
$u$ is isomorphic to the quotient group ${\rm SU}(n+1)/{\rm Image}\big({\mathcal M}_\phi\big)$.

\begin{lemma}
\label{lem:Toda_map}
Suppose that $\phi:\Omega\to {\rm SU}(n+1)$ is a  multi-valued Toda map associated to a solution $u=(u_1,\cdots,u_n)$ of \eqref{equ:Toda_Omega}. Defining an $(n+1)$-tuple $(\hat{f}_0,\cdots,\hat{f}_n)$ of $\mathbb{C}^{n+1}$-multi-valued functions on $\Omega$ by
$$
(\hat{f}_0,\cdots,\hat{f}_n)=\phi\cdot
\begin{pmatrix}
e^{w_0}&&&\\
&e^{w_1}&&\\
&&\ddots&\\
&&&e^{w_n}
\end{pmatrix},
$$
we find that $f_0:=[\hat{f}_0]:\Omega\to {\Bbb P}^n$ is a totally un-ramified unitary curve on $\Omega$
which satisfies $\hat{f}_0\wedge \hat{f}_0^\prime\wedge \hat{f}_0^{(2)}\wedge\cdots\wedge \hat{f}_0^{(n)}=e_0\wedge\cdots\wedge e_n$. Moreover, $(u_1,\cdots,u_n)$ coincides with
the solution of \eqref{equ:Toda_Omega} constructed from the curve $f_0$ by \eqref{equ:sol_Omega}.
\end{lemma}
\begin{proof}
Choose a germ $\varphi$ of $\phi:\Omega\to {\rm SU}(n+1)$. Since
$\frac{\partial \varphi}{\partial \bar z}=\varphi{\mathcal V}$ and 
\quad $\|\hat{f}_i\|=e^{w_i}$, 
it follows from direct computation that the germ $(\hat{\frak f}_0,\cdots,\hat{\frak f}_n)$ of  $(\hat{f}_0,\cdots,\hat{f}_n)$ satisfies 
\begin{eqnarray*}
\frac{\partial}{\partial \bar z}\left(\hat{\frak f}_0,\cdots,\hat{\frak f}_n\right)&=&
\left(0,\,\frac{\|\hat{\frak f}_1\|^2}{\|\hat{\frak f}_0\|^2}\hat{\frak f}_0,\, 
\frac{\|\hat{\frak f}_2\|^2}{\|\hat{\frak f}_1\|^2}\hat{\frak f}_1,
\cdots,\, \frac{\|\hat{\frak f}_n\|^2}{\|\hat{\frak f}_{n-1}\|^2}\hat{\frak f}_{n-1}\right),\quad {\rm and}\\
\frac{\partial}{\partial z}\left(\hat{\frak f}_0,\cdots,\hat{\frak f}_n\right)&=&
\left(\hat{\frak f}_1,\,\cdots, \hat{\frak f}_n,\,0\right)+
\left(\hat{\frak f}_0\frac{\partial}{\partial z}\log\,\|\hat{\frak f}_0\|^2,\,
\hat{\frak f}_1\frac{\partial}{\partial z}\log\,\|\hat{\frak f}_1\|^2,\cdots,\,
\hat{\frak f}_n\frac{\partial}{\partial z}\log\,\|\hat{\frak f}_n\|^2\right).
\end{eqnarray*}
By the first equation above, the germ $\hat{\frak f}_0$ of $\hat{f}_0$ is holomorphic. 
By the second one and induction argument, we obtain that 
\begin{equation}
\label{equ:k_wedge}
\hat{\frak f}_0\wedge\hat{\frak f}_0'\wedge\cdots\wedge \hat{\frak f}_0^{(k)}=
\hat{\frak f}_0\wedge\hat{\frak f}_1\wedge\cdots\wedge \hat{\frak f}_k   \end{equation}
for all $k=0,1,\cdots, n$. In particular, we can see
$\hat{\frak f}_0\wedge\hat{\frak f}_0'\wedge\cdots\wedge \hat{\frak f}_0^{(n)}\equiv e_0\wedge\cdots\wedge e_n$
by using the definition of $(\hat{f}_0,\cdots,\hat{f}_n)$ and $w_0+\cdots w_n=0$. Since $\phi$ has monodromy in 
${\rm SU}(n+1)$, $f_0=[\hat{f}_0]:\Omega\to {\Bbb P}^n$ is a totally un-ramified unitary curve.

Since $\hat{\frak f}_0,\cdots, \hat{\frak f}_n$ are mutually orthogonal, we find by using 
\eqref{equ:k_wedge} that 
\begin{equation}
\label{equ:k_norm}
\|\Lambda_k([\hat{f}_0])\|=\|\hat{\frak f}_0\wedge\hat{\frak f}_0'\wedge\cdots\wedge \hat{\frak f}_0^{(k)}\|
=\|\hat{\frak f}_0\wedge\hat{\frak f}_1\wedge\cdots\wedge \hat{\frak f}_k \|=
\|\hat{\frak f}_0\|\cdot\|\hat{\frak f}_1\|\cdots\|\hat{\frak f}_k\|.\end{equation}
In particular, $\|\Lambda_n([\hat{f}_0])\|=e^{w_0+\cdots+w_n}=1$. Since for all $i=1,\cdots, n$
\begin{eqnarray*}
u_i=2w_i-2w_{i-1}=2\left(\log\,\|\hat{\frak f}_i\|-\log\,\|\hat{\frak f}_{i-1}\|\right),
\end{eqnarray*}
by using \eqref{equ:k_norm} and direct computation, we obtain that $u=(u_1,\cdots, u_n)$ coincides with
the one in \eqref{equ:sol_Omega}.
\end{proof}

\begin{definition}
\label{defn:ass_curve}
We call $f_0:\Omega\to {\Bbb P}^n$ in Lemma \ref{lem:Toda_map} a {\it unitary curve associated with} solution $u=(u_1,\cdots, u_n)$ of the ${\rm SU}(n+1)$ Toda system. The monodromy of $f_0$ is induced by that of the multi-valued Toda map $\phi:\Omega\to {\rm SU}(n+1)$. 
Moreover, such a unitary curve is unique up to a rigid motion in $\big({\Bbb P}^n,\,\omega_{\rm FS}\big)$ 
(\cite[(4.12)]{Griffiths:1974}).
\end{definition}

\section{Canonical unitary curves associated with solutions}

In this section, we shall prove the former part of Theorem  \ref{thm:model} (i), which is restated in the following:

\begin{theorem}
\label{thm:can}
Let $u$ be a solution to \eqref{equ:Toda}. Then there exist
$(n+1)$  holomorphic functions $g_0,\cdots, g_n$ satisfying the normalized condition on $D$ such that 
the following unitary curve 
$[\nu]:D^*\to {\Bbb P}^n,\, z\mapsto \left[z^{\beta_0}g_0(z),\,z^{\beta_1}g_1(z),\,\cdots,\, z^{\beta_n}g_n(z)\right]$
is associated with $u|_{D^*}$ in the sense of Definition \ref{defn:ass_curve}.  We call such $[\nu]$'s {\rm canonical curves} associated with $u$.
\end{theorem}

We cite the following lemma about the ordinary differential equation of $(n+1)$th order given by solution $u$, which was 
discovered by Lin-Wei-Ye \cite{LWY:2012}.

\begin{lemma}
Let $u$ be a solution to \eqref{equ:Toda}, and $f_0=[\hat{f}_0]$ the unitary curve associated with $u|_{D^*}$ obtained by Lemma \ref{lem:curve_Omega}. Then all the $(n+1)$ components of $\hat{f}_0$ form a set of fundamental solutions to the following ordinary differential equation of $(n+1)$th order
\begin{equation}
\label{equ:Fuchs}
y^{(n+1)}+\sum_{k=0}^{n-1}\, Z_{k+1}y^{(k)}=0\quad \text{on}\quad D^*
\end{equation}
which 
satisfies the following two properties{\rm:}
\begin{itemize}
    \item[(i)] The coefficients $Z_k$ are holomorphic on $D^*$ and have poles of order $\leq (n+2-k)$ for all $1\leq k\leq n$. Hence $0$ is the regular singularity of \eqref{equ:Fuchs}.
    \item[(ii)] $\beta_0,\,\beta_1,\,\cdots,\,\beta_n$ defined in \eqref{equ:beta} are the local exponents of  
    \eqref{equ:Fuchs} at $0$.
\end{itemize}
\end{lemma}

\begin{proof} The proof of this lemma is scattered throughout the first, second, fifth, and seventh sections  of 
Lin-Wei-Ye \cite{LWY:2012}.  We sketch it here for completeness. By the proof of \cite[Lemmas 2.1 and 5.2]{LWY:2012}, where
Lin-Wei-Ye used all the conditions in \eqref{equ:Toda}, we obtain that 
$f:=e^{2w_0}=\|\hat{f}_0\|^2$ with $(\hat{f}_0)^{\rm T}(z)=:\nu(z)=\big(\nu_0(z),\,\cdots,\, \nu_n(z)\big)$ satisfies equation \eqref{equ:Fuchs}, i.e.
$L(f)=f^{(n+1)}+\sum_{k=0}^{n-1}\, Z_{k+1}f^{(k)}=0$ on $D^*$,
whose local exponents are $\beta_0,\cdots, \beta_n$.
Hence $0=\overline{L}L(f)=\sum_{i=0}^n\, |L\big(\nu_i(z)\big)|^2$ and 
$L\big(\nu_i(z)\big)=0$ for all $0\leq i\leq n$. On the other hand, since the unitary curve
$f_0:D^*\to {\Bbb P}^n$ is totally un-ramified and then non-degenerate, $\nu_0(z),\cdots,\nu_{n-1}(z)$ and $\nu_n(z)$ are $(n+1)$ multi-valued holomorphic functions whose germs at each point of $D^*$ are linearly independent over ${\Bbb C}$. 
Hence $\{\nu_i\}_{i=0}^n$ is a set of fundamental solutions of \eqref{equ:Fuchs}.
\end{proof}

{\sc Proof of Theorem \ref{thm:can}}:  It suffices to show that {\it there exists a matrix $A$ in 
${\rm SU}(n+1)$ such that 
$\nu(z) A:=\big(\hat{f}_0(z)\big)^{\rm T} A=
\left(z^{\beta_0}g_0(z),\,z^{\beta_1}g_1(z),\,\cdots,\, z^{\beta_n}g_n(z)\right)$
for some $(n+1)$ holomorphic functions $g_0,\cdots, g_n$,} which satisfy the normalized condition automatically since $\Lambda_n(\nu)=\Lambda_n(\nu\cdot A)\equiv e_0\wedge e_1\wedge\cdots\wedge e_n$ on $D^*$.
We prove it via the following two steps.

{\it Step 1.}  Choose base point $B\in D^*$ and generator $\gamma_B$ of $\pi_1(D^*,\, B)$. 
%with $\left\{\gamma_B(t)=e^{2\pi\sqrt{-1}t}B:t\in [0,\, 1]\right\}$. 
Since for each $A\in {\rm SU}(n+1)$, the unitary curve $[\nu\cdot A]: D^*\to {\Bbb P}^n$ is also associated with $u$ and has monodromy representation conjugate to that of $[\nu]$ by $A$, 
we assume without loss of generality that the monodromy representation $\mathcal{M}_\nu$ of $\nu$ maps $\gamma_B$ to the  diagonal matrix ${\rm diag}\,\left(e^{2\pi\sqrt{-1}b_0} ,\, e^{2\pi\sqrt{-1}b_1} \,\,\cdots,\,e^{2\pi\sqrt{-1}b_n}  \right)\in {\rm SU}(n+1)$ with $b_0,\cdots, b_n\in {\Bbb R}$.
Hence there exist holomorphic functions $ \psi_0,\psi_1,\cdots,\psi_n $ on $D^*$ such that 
\begin{equation}
\label{equ:unitary}
\nu(z)=\big(\nu_0(z),\, \nu_1(z),\,\cdots,\,\nu_n(z)\big)
=\left(z^{b_0}\psi_0(z),\, z^{b_1}\psi_1(z),\, \cdots,\, z^{b_n}\psi_n(z)\right).
\end{equation}

{\it Step 2.} We divide $\beta_0<\beta_1<\cdots<\beta_n$ into the 
following $k$ groups
\begin{equation*} \beta_1^{(1)},\beta_2^{(1)},...,\beta_{i_1}^{(1)};\quad\beta_1^{(2)},\beta_2^{(2)},...,\beta_{i_2}^{(2)};
\quad\cdots;\quad\beta_1^{(k)},\beta_2^{(k)},...,\beta_{i_k}^{(k)} \end{equation*}
such that in one of these groups, each local exponent differs from the other
by integers and the local exponents are in strictly ascending order; 
and any two local exponents lying in different groups are mutually distinct modulo ${\Bbb Z}$. 
Recall that $\beta_0<\beta_1< \cdots<\beta_n$ are all the local exponents of the $(n+1)$ order linear differential equation \eqref{equ:Fuchs},
of which $0$ is a regular singularity. By using the Frobenius method \cite[Section 3.4.1]{Kohno:1999}, we have the following set
of fundamental solutions of this equation on $D^*$:
{\small 
\begin{equation}
\label{equ:Frob}
	\left\{\begin{array}{l}
		y_1^{(1)}(z,\beta_{i_1}^{(1)})=z^{\beta_{i_1}^{(1)}}Y\big(z,\beta_{i_1}^{(1)}\big)\\
		y_2^{(1)}(z,\beta_{{i_1}-1}^{(1)})=z^{\beta_{{i_1}-1}^{(1)}}\left(\frac{\partial}{\partial \beta}Y\big(z,\beta_{{i_1}-1}^{(1)}\big)+Y\big(z,\beta_{{i_1}-1}^{(1)}\big)\log z \right)\\
		\vdots\\
		y_{i_1}^{(1)}(z,\beta_1^{(1)})=z^{\beta_1^{(1)}}\left(\frac{\partial^{{i_1}-1}}{\partial \beta^{{i_1}-1}}Y\big(z,\beta_1^{(1)}\big)+
		C_{{i_1}-1}^1 \frac{\partial^{{i_1}-2}}{\partial \beta^{{i_1}-2}}Y\big(z,\beta_1^{(1)}\big)\log z+\cdots+Y\big(z,\beta_1^{(1)}\big)(\log z)^{{i_1}-1}\right)\\
		\vspace{1pt}\\
		y_1^{(2)}(z,\beta_{i_2}^{(2)})=z^{\beta_{i_2}^{(2)}}Y\big(z,\beta_{i_2}^{(2)}\big)\\
		y_2^{(2)}(z,\beta_{{i_2}-1}^{(2)})=z^{\beta_{{i_2}-1}^{(2)}}\left(\frac{\partial}{\partial \beta}Y\big(z,\beta_{{i_2}-1}^{(2)}\big)+Y\big(z,\beta_{{i_2}-1}^{(2)}\big)\log z \right)\\
		\vdots\\
		y_{i_2}^{(2)}(z,\beta_1^{(2)})=z^{\beta_1^{(2)}}\left(\frac{\partial^{{i_2}-1}}{\partial \beta^{{i_2}-1}}Y\big(z,\beta_1^{(2)}\big)+
		C_{{i_2}-1}^1 \frac{\partial^{{i_2}-2}}{\partial \beta^{{i_2}-2}}Y\big(z,\beta_1^{(2)}\big)\log z+\cdots+Y\big(z,\beta_1^{(2)}\big)(\log z)^{{i_2}-1}\right)\\
		\vspace{1pt}\\ \vdots\\ \vspace{1pt}\\
		y_1^{(k)}(z,\beta_{i_k}^{(k)})=z^{\beta_{i_k}^{(k)}}Y\big(z,\beta_{i_k}^{(k)}\big)\\
		y_2^{(k)}(z,\beta_{{i_k}-1}^{(k)})=z^{\beta_{{i_k}-1}^{(k)}}\left(\frac{\partial}{\partial \beta}Y\big(z,\beta_{{i_k}-1}^{(k)}\big)+Y\big(z,\beta_{{i_k}-1}^{(k)}\big)\log z \right)\\
		\vdots\\
		y_{i_k}^{(k)}(z,\beta_1^{(k)})=z^{\beta_1^{(k)}}\left(\frac{\partial^{{i_k}-1}}{\partial \beta^{{i_k}-1}}Y\big(z,\beta_1^{(k)}\big)+
		C_{{i_k}-1}^1 \frac{\partial^{{i_k}-2}}{\partial \beta^{{i_k}-2}}Y\big(z,\beta_1^{(k)}\big)\log z+\cdots+Y\big(z,\beta_1^{(k)}\big)(\log z)^{{i_k}-1}\right)\\
	\end{array}\right.\, ,
\end{equation}
}
where $ Y(z,\beta) $ is holomorphic with respect to both $ z $ and $ \beta $, and 
\[\frac{\partial^{i_j-m_j}}{\partial \beta^{i_j-m_j}}Y\big(0,\beta_{m_j}^{(j)}\big) \neq 0\quad \text{for all}\quad 
1\leq j\leq k\quad \text{and}\quad 
1\leq m_j\leq i_j.\]
For all $0\leq \ell\leq n$, since $b_\ell\in {\Bbb R}$, all the germs of the multi-valued holomorphic function $\nu_\ell(z)=z^{b_\ell}\psi_\ell(z)$ have the same norm and are uniformly bounded at $B$.
Since each function $y_{m_j}^{(j)}\big(z,\,\beta_{m_j}^{(j)}\big)$ in \eqref{equ:Frob} is a complex linear combination of $\nu_0(z),\cdots, \nu_n(z)$, 
all its germs are also uniformly bounded at $B$.
This precludes the potential presence of any logarithmic terms within these germs, as discussed in Remark \ref{rem:global_analytic}.
That is, for all $ j=1,2,\cdots,k$ and $ m_j=1,2,\cdots,i_j $, the functions in \eqref{equ:Frob} actually have form  $y_{m_j}^{(j)}\big(z,\,\beta_{m_j}^{(j)}\big)=z^{\beta_{m_j}^{(j)}}\phi_{m_j}^{(j)}(z)$, 
where  $ \phi_{m_j}^{(j)}(z) $ are holomorphic  on $D$ such that $ \phi_{m_j}^{(j)}(0) \neq 0 $. 
%由此可知，任何一个$ \nu_i $都可写成形如$ \nu_i(z)=\sum\limits_{j=0}^m a_j(z)(logz)^j $的形式，其中$ a_j(z) $属于由（单值）全纯函数族和幂函数族所生成的环。现在假设$ m >0 $，将自变量$ z $以$ B $绕原点转n圈，则$ logz $将变为$ logz+2\pi n \sqrt{-1} $，因此在任意区域上当n无限增大时多值函数$ \nu_i(z) $相应的分支都是无界的。然而，事实上$\nu_i(z)= z^{b_i}\psi_i(z) $，以任意点为基点绕原点转一周函数值的模都保持不变，因此多值函数$ \nu_i(z) $在闭圆环$ r_1 \le |z| \le r_2 <1 $上是一致有界的，矛盾！
%\begin{equation} \beta_1^{(1)},\beta_2^{(1)},...\beta_{i_1}^{(1)};\beta_1^{(2)},\beta_2^{(2)},...\beta_{i_2}^{(2)};\cdots;\beta_1^{(k)},\beta_2^{(k)},...\beta_{i_k}^{(k)} \end{equation}
%is the rearrangement and renumber of $ \beta_0, \beta_1, \cdots,\beta_n $, s.t. two indicial exponents in one group with a same superior character differ by an integer, vice versa; and $\beta_1^{(j)}<\beta_2^{(j)}<\cdots<\beta_{i_j}^{(j)}, j=1,2,\cdots,k$. Of course, $ i_1+i_2+\cdots +i_k=n+1 $. While $ Y(z,\beta) $ is a holomorphic function of $ z $ and $ \beta $, and $ \frac{\partial^{i_j-m}}{\partial \beta^{i_j-m}}Y(0,\beta_m^{(j)}) \neq 0 $ for $ j=1,2,\cdots,k; \; m=1,2,\cdots,i_j $.
Hence, there exists $M\in {\rm GL}(n+1,\,{\Bbb C})$ such that 
\begin{eqnarray*}
\nu(z)&=&\big(\nu_0(z),\, \nu_1(z),\,\cdots,\,\nu_n(z)\big)=\big(z^{b_0}\psi_0(z),\, z^{b_1}\psi_1(z),\, \cdots,\, z^{b_n}\psi_n(z)\big)\\
&=&\big( y_1^{(1)},y_2^{(1)},...,y_{i_1}^{(1)};\, y_1^{(2)},y_1^{(2)},...,y_{i_2}^{(2)};\,\cdots;\, y_1^{(k)},y_2^{(k)},...,y_{i_k}^{(k)} \big)\cdot M.
\end{eqnarray*}
Given $1\leq j\leq k$, the multi-valued functions $y_{m_j}^{(j)}(z)$ have the same monodromy mapping $\gamma_B$ to multiple 
$e^{2\pi\sqrt{-1}\beta_{m_j}^{(j)}}$ for all $m_j=1,2,\cdots,i_j$.
Recall that the monodromy of the set $ (\nu_0,\, \nu_1,\, \cdots,\,\nu_n) $ of fundamental solutions 
to \eqref{equ:Fuchs} maps $\gamma_B$ to ${\rm diag}\,\left(e^{2\pi\sqrt{-1}b_0} ,\, e^{2\pi\sqrt{-1}b_1} \,\,\cdots,\,e^{2\pi\sqrt{-1}b_n}  \right)$. Adjusting the order of 
$\nu_0(z),\cdots, \nu_n(z)$ if necessary,  we obtain that
$\big(\nu_0(z),\, \nu_1(z),\,\cdots,\,\nu_n(z)\big)$ equals \begin{eqnarray*}
\big( y_1^{(1)},y_2^{(1)},...,y_{i_1}^{(1)};\, y_1^{(2)},y_1^{(2)},...,y_{i_2}^{(2)};\,\cdots;\, y_1^{(k)},y_2^{(k)},...,y_{i_k}^{(k)} \big)\cdot 
{\rm diag}(C_1,\cdots, C_k), 
\end{eqnarray*}
for some $C_1,\cdots, C_k\in {\rm GL}(i_j,\, {\Bbb C})$.
%$$M={\rm diag}(C_1,\cdots, C_k),\quad {\rm where}\quad C_j\in {\rm GL}(i_j,\, {\Bbb C}).$$
%it can be derived that all of the non-zero terms composing a certain $ f_i $ in (4.5) have the same superscript.\\
%In summary, rearraging and rewriting $ (f_0,f_1,\cdots,f_n) $ similarly to (4.4), we can obtain:
%\begin{equation}\begin{split}
		%&\quad(f_1^{(1)},f_2^{(1)},\cdots,f_{i_1}^{(1)};f_1^{(2)},f_2^{(2)},\cdots,f_{i_2}^{(2)};......;f_1^{(k)},f_2^{(k)},\cdots,f_{i_k}^{(k)}) \\
		%&=\Big(z^{\beta_1^{(1)}}\phi_1^{(1)}(z),z^{\beta_2^{(1)}}\phi_2^{(1)}(z),\cdots,z^{\beta_{i_1}^{(1)}}\phi_{i_1}^{(1)}(z);
		%z^{\beta_1^{(2)}}\phi_1^{(2)}(z),z^{\beta_2^{(2)}}\phi_2^{(2)}(z),\cdots,z^{\beta_{i_2}^{(2)}}\phi_{i_2}^{(2)}(z);\\
		%& \quad ......;
		%z^{\beta_1^{(k)}}\phi_1^{(k)}(z),z^{\beta_2^{(k)}}\phi_2^{(k)}(z),\cdots,z^{\beta_{i_k}^{(k)}}\phi_{i_k}^{(k)}(z)\Big)
		%\cdot \begin{pmatrix}
		%	C_1 &  &  & \\
		%	& C_2 &  & \\
		%	&  & \ddots  & \\
		%	&  &  & C_k
		%\end{pmatrix},
%\end{split}\end{equation}
%where $ C_j $ is a non-singular complex matrix of $ i_j \times i_j , \quad j=1,2,\cdots,k $.
For all $j=1,\cdots, k$, we rewrite $C_j$ as the product $C_j=B_j A_j $, where $ B_j $ is a lower triangular matrix and $ A_j $ is a unitary matrix. Recalling $\beta_1^{(j)}<\beta_2^{(j)}<\cdots<\beta_{i_j}^{(j)}$, 
we may assume that the lower triangular matrix $B_j=I_{i_j}$ and $C_j=A_j$ since 
$$ \Big(z^{\beta_1^{(j)}}\phi_1^{(j)}(z),z^{\beta_2^{(j)}}\phi_2^{(j)}(z),\cdots,z^{\beta_{i_j}^{(j)}}\phi_{i_j}^{(j)}(z)\Big) \cdot B_j=\Big(z^{\beta_1^{(j)}}g_1^{(j)}(z),\, z^{\beta_2^{(j)}}g_2^{(j)}(z),\,\cdots,\, z^{\beta_{i_j}^{(j)}}g_{i_j}^{(j)}(z)\Big),$$ 
where $g_1^{(j)}(z),\,\cdots,\,g_{i_j}^{(j)}(z)$ are holomorphic functions on $D$ non-vanishing at $0$. 
We are done by taking 
$A={\rm diag}(A_1,\cdots, A_k)$. $ \hfill{\Box} $

\section{Completion of the proof for Theorem 1.2.}
In the preceding section, we proved an important part of Theorem 1.2., i.e. the canonical unitary curve $[\nu(z)]=[\nu_0(z),\cdots, \nu_n(z)]$ is associated with solution
$u=(u_1,\cdots, u_n)$ to \eqref{equ:Toda}. We shall complete the proof
of the theorem in this section by applying both the infinitesimal Pl\" ucker formulae and the $D^*$-case of \eqref{equ:sol_Omega} 
to $\nu(z)$ and its associated curves 
$\nu(z)\wedge\nu'(z)\wedge\cdots\wedge \nu^{(k)}(z)$ for all $k=1,\cdots, n$.
Here we recall that 
$\nu(z)\wedge\nu'(z)\wedge\cdots\wedge \nu^{(n)}(z)\equiv e_0\wedge\cdots\wedge e_n$.
To this end, we prepare a lemma relevant to linear algebra in the following:

\begin{lemma}
\label{lem:normalize}
Let $g_0(z),\cdots,g_k(z) $ be holomorphic functions on $D$ where $0\leq k\leq n$. 
Then
there exists a holomorphic function 
\[G_k=G_k\big(\beta_0,\beta_1,\cdots,\beta_k;g_0(z),g_1(z),\cdots,g_k(z);z\big)\]
in $D$ such that 
\begin{equation}\begin{split}
			\begin{vmatrix}
				z^{\beta _0}g_0(z)   & z^{\beta _1}g_1(z) & \cdots  & z^{\beta_k}g_n(z) \\
				\big(z^{\beta_0}g_0(z)\big)' & \big(z^{\beta_1}g_1(z)\big)' & \cdots  & \big(z^{\beta_n}g_k(z)\big)' \\
				\vdots  & \vdots  & \cdots  & \vdots \\
				\big(z^{\beta_0}g_0(z)\big)^{(k)}  & \big(z^{\beta_1}g_1(z)\big)^{(k)} & \cdots  & \big(z^{\beta_k}g_k(z)\big)^{(k)}
			\end{vmatrix}\\
			=z^{\sum\limits_{i=0}^{k} \beta_i-\frac{k(k+1)}{2}  }\cdot
			G_{k}\big(\beta_0,\beta_1,\cdots,\beta_k;g_0(z),g_1(z),\cdots,g_k(z);z\big).
	\end{split}\end{equation}
	In particular, 
		$G_k|_{z=0} =\prod\limits_{i=0}^{k} g_i(0) \cdot \prod\limits_{0\le i<j \le k} (\beta_i- \beta_j)$.
\end{lemma}
\begin{proof}
	By the Leibnitz rule, for all $0\leq \ell\leq k$, we have $\big(z^{\beta_i}g_i(z)\big)^{(\ell)}=z^{\beta_i-\ell}\cdot g_{i\ell}(z)$,  where 
	$
g_{i\ell}(z)=\beta_i (\beta_i-1)\cdots (\beta _i-\ell+1)g_i(z)
		+C_\ell^{1} \beta_i (\beta_i-1)\cdots (\beta_i-\ell+2)zg_i'(z)+
		\cdots + z^\ell g_i^{(\ell)}(z)
	$
	are holomorphic functions on $D$. Therefore, there holds
\begin{eqnarray*}
  &&\begin{vmatrix}
		z^{\beta_0}g_0(z) & z^{\beta_1}g_1(z) & \cdots  & z^{\beta_k}g_k(z) \\
		\big(z^{\beta_0}g_0(z)\big)' & \big(z^{\beta_1}g_1(z)\big)' & \cdots  & \big(z^{\beta_k}g_k(z)\big)' \\
		\vdots  & \vdots  & \cdots  & \vdots \\
		\big(z^{\beta_0}g_0(z)\big)^{(k)}  & \big(z^{\beta_1}g_1(z)\big)^{(k)} & \cdots  & \big(z^{\beta_k}g_k(z)\big)^{(k)}
	\end{vmatrix}
 =\begin{vmatrix}
		z^{\beta_0}g_0(z)   & z^{\beta_1}g_1(z) & \cdots  & z^{\beta_k}g_k(z) \\
		z^{\beta_0-1}g_{01}(z) & z^{\beta_1-1}g_{11}(z) & \cdots  & z^{\beta_k-1}g_{k1}(z) \\
		\vdots   & \vdots  & \cdots  & \vdots  \\
		z^{\beta_0-k}g_{0k}(z) & z^{\beta_1-k}g_{1k}(z) & \cdots  & z^{\beta_k-k}g_{kk}(z)
	\end{vmatrix} \\
&=& z^{\sum\limits_{i=0}^k \beta_i-\frac{k(k+1)}{2}}\cdot \begin{vmatrix}
		g_0(z)  & g_1(z) & \cdots  & g_k(z) \\
		g_{01}(z) & g_{11}(z) & \cdots  & g_{k1}(z) \\
		\vdots   & \vdots  & \cdots  & \vdots  \\
		g_{0k}(z) & g_{1k}(z) & \cdots  & g_{kk}(z)
	\end{vmatrix}=:z^{\sum\limits_{i=0}^k \beta_i-\frac{k(k+1)}{2}}G_k, \end{eqnarray*}
	where 
 $G_k=G_k\big(\beta_0,\beta_1,\cdots,\beta_k;g_0(z),g_1(z),\cdots,g_k(z);z\big):= \begin{vmatrix}
		g_0(z)  & g_1(z) & \cdots  & g_k(z) \\
		g_{01}(z) & g_{11}(z) & \cdots  & g_{k1}(z) \\
		\vdots   & \vdots  & \cdots  & \vdots  \\
		g_{0k}(z) & g_{1k}(z) & \cdots  & g_{kk}(z)
	\end{vmatrix}$.
Finally,   we find by the preceding expression of $g_{i\ell}(z)$  that $ G_k|_{z=0}$ equals
\begin{eqnarray*}
&&\begin{vmatrix}g_0(0)   & g_1(0) & \cdots  & g_n(0) \\
		\beta_0 g_0(0) & \beta_1 g_1(0) & \cdots  & \beta_n g_n(0) \\
		\vdots   & \vdots  & \cdots  & \vdots  \\
		\beta_0 (\beta_0-1)\cdots (\beta_0-k+1)g_0(0) &  \beta_1 \cdots (\beta_1-k+1) g_1(0)
		& \cdots  &  \beta_n \cdots (\beta_n-k+1) g_n(0)
	\end{vmatrix}\\
	&=&g_0(0) g_1(0)\cdots g_n(0)\begin{vmatrix}
		1 & 1 & \cdots  & 1\\
		\beta_0 & \beta_1 & \cdots  & \beta_n \\
		\vdots  & \vdots  & \cdots  & \cdots  \\
		\beta_0^n  & \beta_1^n & \cdots  & \beta_n^n
	\end{vmatrix}
	=\prod\limits_{i=0}^n g_i(0) \cdot \prod\limits_{0 \le i< j\le n}(\beta_i-\beta_j). 
 \end{eqnarray*}
\end{proof}

Using this lemma, we obtain the following three formulae relevant to the lifting  $\nu(z)=\big(z^{\beta_0}g_0(z),z^{\beta_1}g_1(z),\cdots,z^{\beta_n}g_n(z)\big)$ of the canonical curve $[\nu(z)]$ 
associated with solution $u$ to \eqref{equ:Toda}.

\begin{formula} {\it For all $k=0,1,\cdots. n$, we have
\begin{align}
	\Lambda_k(z) & =\nu(z)\wedge \nu '(z)\wedge\cdots \wedge \nu^{(k)}(z)\nonumber\\
		&=\sum_{0 \le i_0< i_1< \cdots < i_k\le n } \begin{vmatrix}
			z^{\beta_{i_0}}g_{i_0}(z)   & z^{\beta_{i_1} }g_{i_1}(z)  & \cdots  & z^{\beta_{i_k}}g_{i_k}(z)  \\
			\big(z^{\beta_{i_0}}g_{i_0}(z)\big)'  & \big(z^{\beta_{i_1}}g_{i_1}(z)\big)' & \cdots  & \big(z^{\beta_{i_k}}g_{i_k}(z)\big)' \\
			\vdots  & \vdots  & \cdots  & \vdots \\
			\big(z^{\beta_{i_0} }g_{i_0}(z)\big)^{(k)}  & \big(z^{\beta_{i_1} }g_{i_1}(z)\big)^{(k)} & \cdots   & \big(z^{\beta _{i_k}}g_{i_k}(z)\big)^{(k)}
		\end{vmatrix}
		e_{i_0}\wedge  e_{i_1}\wedge \cdots \wedge e_{i_k} \nonumber\\
		&=\sum_{0 \le i_0< i_1< \cdots < i_k\le n } z^{\sum\limits_{j=0}^k \beta_{i_j}-\frac{k(k+1)}{2}}
		\cdot G_k\big(\beta_{i_0},\beta_{i_1},\cdots,\beta_{i_k};g_{i_0}(z),g_{i_1}(z),\cdots g_{i_k}(z);z\big)
		e_{i_0}\wedge e_{i_1}\wedge \cdots \wedge e_{i_k} \nonumber	
\end{align}
Recall that $g_0,\cdots, g_n$ satisfy the normalized condition \eqref{equ:normalized} and
$\beta_0+\cdots+\beta_n=n(n+1)/2$, which implies that $\Lambda_n(z)=e_0\wedge\cdots\wedge e_n$.}
\end{formula}
\begin{proof}
It follows from a straightforward computation by using the very expression $\big(z^{\beta_0}g_0(z),z^{\beta_1}g_1(z),\cdots,z^{\beta_n}g_n(z)\big)$ of $\nu(z)$.    
\end{proof}

By the definition of $\beta_0<\beta_1<\cdots\beta_n$, the summand with the lowest degree with respect to $z$ on the left hand side of Formula 1 
has form 
\begin{eqnarray*}
&& z^{\sum\limits_{j=0}^k \beta_j-\frac{k(k+1)}{2}}\cdot
G_{k}\big(\beta_0,\beta_1,...\beta_k;g_0(z),g_1(z),\cdots,g_k(z);z\big)e_0\wedge  e_1\wedge \cdots \wedge e_k \\
&=&z^{-\alpha_{k+1}}G_{k}\big(\beta_0,\beta_1,...\beta_k;g_0(z),g_1(z),\cdots,g_k(z);z\big)e_0\wedge  e_1\wedge \cdots \wedge e_k.
\end{eqnarray*}
Hence, $\Lambda_k(z)$ equals 
$z^{-\alpha _{k+1}}\,G_k\big(\beta_0,\beta_1,\cdots ,\beta_k;
g_0(z),g_1(z),\cdots ,g_k(z);z\big) e_0\wedge \cdots \wedge e_k$ plus 
\begin{eqnarray*}
\sum\limits_{0 \le i_0 < i_1< \cdots < i_k\le n \atop i_k> k } z^{\sum\limits_{j=0}^k \beta_{i_j}-\frac{k(k+1)}{2}}
\cdot G_k\big(\beta_{i_0},\beta_{i_1},\cdots,\beta_{i_k};g_{i_0}(z),g_{i_1}(z),\cdots g_{i_k}(z);z\big)
e_{i_0}\wedge e_{i_1}\wedge \cdots \wedge e_{i_k}.
\end{eqnarray*}
Hence, we reach the last two formulae in the following:

\begin{formula}
	$\left \| \Lambda_k \right \|^2 =\left |z \right | ^{-2\alpha_{k+1} }
	\bigg(\left | G_k \big(\beta_0, \beta_1,\cdots ,\beta_k;
	g_0(z),g_1(z),\cdots ,g_k(z);z\big) \right |^2 $
	\begin{equation*}
		+\sum_{0 \le i_0 < i_1< \cdots < i_k\le n \atop i_k> k }
		\left | z \right | ^{2(\sum\limits_{j=0}^k \beta_{i_j}-\frac{k(k+1)}{2}+\alpha_{k+1})}
		\left | G_k\big(\beta_{i_0},\beta_{i_1},\cdots,\beta_{i_k};g_{i_0}(z),g_{i_1}(z),\cdots,g_{i_k}(z);z\big)
		\right |^2\bigg).	
\end{equation*}
In particular, Therefore, $ \log\left \| \Lambda _k \right \|^2$ equals $-2\alpha_{k+1}\log\,|z|$ plus a H\" older continuous function
near $0$ by Lemma \ref{lem:normalize}.
\end{formula}

%Therefore, $ \log\left \| \Lambda _k \right \|^2 =-2\alpha_{k+1}log|z| +O(1) $, and then

\begin{formula}
For all $k=1,2,\cdots, n$, we have
\begin{align}
		u_k=-\sum_{j=1}^n a_{kj}\log\left \| \Lambda_{j-1} \right \| ^2 =
		-\sum_{j=1}^{n}a_{kj}\big(-2\alpha_{j} \log \left | z \right |+O(1)\big)=
		2\gamma_k \log|z|+O(1),  \end{align}
	where 
\begin{eqnarray*}
     O(1)&=&-\sum\limits_{j=1}^n a_{kj}\log\,\Bigg(\left | G_{j-1} \big(\beta_0, \beta_1,\cdots ,\beta_{j-1} ;
	g_0(z),g_1(z),\cdots ,g_{j-1}(z);z\big) \right | ^{2}+ \\
	&\sum\limits_{0 \le i_0 < i_1< \cdots < i_{j-1}\le n \atop i_{j-1}> j-1 }&
	\left | z \right | ^{2(\sum\limits_{l=0}^{j-1} \beta _{i_l}-\frac{(j-1)j}{2}+\alpha_j)}
	\left | G_{j-1}\big(\beta _{i_0},\beta _{i_1},\cdots,\beta_{i_{j-1}};g_{i_0}(z),g_{i_1}(z),\cdots,g_{i_{j-1}}(z);z\big)
	\right |^2\Bigg) \end{eqnarray*}
 is a H\" older continuous function near $0$ by Lemma \ref{lem:normalize}. 
\end{formula}

{\sc Proof of Theorem 1.2 (i)}\quad Formula 3 coincides with the second sentence of Theorem 1.2. (i). 
As long as the last sentence is concerned, any unitary curve $D^*\to {\Bbb P}^n$ with form \eqref{equ:can} induces a solution
$u=(u_1,\cdots, u_n)$ to the ${\rm SU}(n+1)$ system in $D^*$. 
By Formula 3, $u_k$ equals $2\gamma_k \log|z|$ plus a bounded smooth function near $0$ for all $k=1,\cdots,n$. Hence, by Formula 3 and the infinitesimal Pl\" ucker formula,  $u=(u_1,\cdots, u_n)$ satisfies the system of PDEs in \eqref{equ:Toda} provided that 
the original integral condition in \eqref{equ:Toda} is replaced by the local integrability of
 $e^{u_k}$ in $D$. QED  \\

{\sc Proof of Theorem 1.2 (ii)} The idea of the proof goes as follows: 
by using some complex coordinate transformation 
$z\mapsto \xi(z)$ near $0$ and preserving $0$, we could simplify further the expression of the canonical curve $ \nu(z) $ under the new coordinate $\xi$. Then we obtain the desired form  for the K\" ahler metric $\frac{\sqrt{-1}}{2\pi}\,e^{u_1}\,{\rm d}z\wedge {\rm d}\bar z$ on $D^*$,
which coincides with the pull-back metric
$[\nu]^*\big(\omega_{\rm FS}\big)$ by \eqref{equ:Plucker} and \eqref{equ:sol_Omega}.
The details consist of the following three steps.

{\it Step 1.}  Recall that $\nu(z)=
\big(z^{\beta_0}g_0(z),z^{\beta_1}g_1(z),\cdots,z^{\beta_n}g_n(z)\big)$ where $g_0,\cdots, g_n$ satisfy the normalized condition
so that $g_0(0)g_1(0)\cdots g_n(0)\not=0$. Then we choose the new complex coordinate 
$\xi=z\cdot \left(\frac{g_1(z)}{g_0(z)}\right)^{\frac{1}{\beta_1 -\beta_0}}$
near $z=0$ and preserving $0$. Then, under this new coordinate $\xi$, there exist $(n-1)$ holomorphic functions
$\tilde g_2(\xi),\cdots, \tilde g_n(\xi)$ near $0$ and non-vanishing at $0$ such that $\nu$ has the simpler form of
\begin{equation}
\label{equ:simple}
\tilde \nu(\xi):=\nu\big(z(\xi)\big)=\Big(\xi^{\beta_0},\, \xi^{\beta_1},\, \xi^{\beta_2}\tilde g_2(\xi),\, \cdots,\, 
\xi^{\beta_n}\tilde g_n(\xi)\Big).
\end{equation}

{\it Step 2.} The preceding curve $\tilde \nu(\xi)$ does not satisfy the normalized condition with respect to $\xi$ near 
$0$ in general, which will not bring us trouble since the pull-back metric 
$\frac{\sqrt{-1}}{2\pi}\,e^{u_1}\,{\rm d}z\wedge {\rm d}\bar z=[\nu]^*(\omega_{\rm FS})$ 
is invariant under the coordinate transformation. On one hand, by using Formula 3, we have
\begin{eqnarray*}
[\nu]^*\big(\omega_{\rm FS}\big)&=&
\frac{\left | G_1\big(\beta_0,\beta_1;g_0(z), g_1(z);z\big) \right |^2
		+ \sum\limits_{0\le i_0< i_1\le n \atop i_1> 1}  |z|^{2(\beta_{i_0}+ \beta_{i_1}-1+\alpha_2)}
		\left| G_1\big(\beta_{i_0},\beta_{i_1};g_{i_0}(z), g_{i_1}(z);z\big)\right|^2}
	{\Big(\left | g_0(z) \right |^2+ \left | z \right |^{2(\beta _1-\beta_0)} \left | g_1(z) \right |^2
		+\cdots +\left | z \right |^{2(\beta_n-\beta_0)} \left | g_n(z) \right |^2 \Big)^2}\nonumber\\
&&\cdot \frac{\sqrt{-1}}{2\pi} \left | z \right | ^{2\gamma_1} dz\wedge d\bar{z}.
\end{eqnarray*}
On the other hand, substituting the simpler form \eqref{equ:simple} of $\nu(z)$ to the preceding equality, we could simplify
the pull-back metric $[\nu]^*\big(\omega_{\rm FS}\big)$ to the form of 
\begin{equation*}
\left | \xi \right | ^{2\gamma_1}\frac{(\beta_1-\beta_0)^2
		+\sum\limits_{0\le i_0< i_1\le n \atop i_1> 1} \left | \xi \right |^{2(\beta_{i_0}+\beta_{i_1}-1+\alpha_2)}
		\left | G_1\big(\beta_{i_0},\beta_{i_1};\tilde{g}_{i_0}(\xi),\tilde{g}_{i_1}(\xi);\xi \big) \right |^2 }
	{\Big(1+\left | \xi \right |^{2(\beta_1-\beta_0)}+\left | \xi \right |^{2(\beta_2-\beta_0)}\left |\tilde{g}_2(\xi) \right |^2
		+\left | \xi \right |^{2(\beta_n-\beta_0)}\left |\tilde{g}_n(\xi) \right |^2 \Big)^2 }
	\frac{\sqrt{-1}}{2\pi}{\rm d}\xi\wedge {\rm d}\bar{\xi}.
\end{equation*}
In particular, the pull-back metric $[\nu]^*\big(\omega_{\rm FS}\big)$  has cone singularity at $0$ with angle $2\pi(1+\gamma_1)$.

{\it Step 3.} Since the K\" ahler metric $\frac{\sqrt{-1}}{2\pi}\,e^{u_k}\,{\rm d}z\wedge {\rm d}\bar z$ on $ D^*$ 
coincides with the pull-back metric 
$[\nu\wedge\nu'\wedge\cdots\wedge \nu^{(k-1)}]^*\big(\omega_{\rm FS}\big)=
\frac{\sqrt{-1}}{2\pi}\,\frac{\|\Lambda_{k-2}(\nu)\|^2\cdot\|\Lambda_{k}(\nu)\|^2 }{\|\Lambda_{k-1}(\nu)\|^4}\, {\rm d}z\wedge {\rm d}\bar z$
by \eqref{equ:Plucker} and \eqref{equ:sol_Omega} for all $k=2,3,\cdots, n$,  this metric has cone singularity at $0$ of angle $2\pi(1+\gamma_k)$ and could be simplified correspondingly by using both Formula 3 and \eqref{equ:simple}. $ \hfill{\Box} $

\begin{remark}
In the case of $n\geq 2$, a clear distinction can be observed between the classification of finite-energy solutions for the ${\rm SU}(n+1)$ Toda system on ${\Bbb C}\backslash \{0\}$ as delineated by Lin-Wei-Ye \cite[Theorem 1.1.]{LWY:2012}, and our own in Theorem \ref{thm:model}. Lin-Wei-Ye's classification involves a finite number of parameters, while our method requires $(n-1)$ non-vanishing holomorphic functions in the vicinity of $0$, incorporating infinitely many parameters, even after applying necessary coordinate changes near $0$. This discrepancy emerges from the fact that Lin-Wei-Ye were able to express finite-energy solutions on ${\Bbb C}\backslash \{0\}$ in terms of {\it unitary curves} $\nu(z)=[z^{\beta_0},\cdots, z^{\beta_n}]A$ on ${\Bbb C}\backslash \{0\}$, where $A$ are suitable automorphisms of ${\Bbb P}^n$ \cite[pp. 189-190]{LWY:2012}. 
The selection of A depends on the situation where numbers like $\gamma_i+\cdots+\gamma_j$
are rounded to integers. The diversity of all possible choices for A determines the number of parameters in the Lin-Wei-Ye classification (\cite[Theorem 1.1.]{LWY:2012}).
In particular, the parameter count in their classification reaches its maximum of $n(n+2)$ precisely when all $\gamma_j$ are non-negative integers. In such instances, given that $\beta_0,\cdots,\beta_n$ are all integers, we can construct all solutions using unitary curves $\nu(z)=[z^{\beta_0},\cdots, z^{\beta_n}]A$ with any automorphisms $A$ of ${\Bbb P}^n$.
\end{remark}

\noindent\textbf{Acknowledgements:}
B.X. expresses sincere gratitude to Professor Guofang Wang at the University of Freiburg for introducing him to the field of Toda systems during the summer of 2018 and providing valuable references in the spring of 2023. Special thanks are extended to Professor Zhijie Chen at Tsinghua University, who, in November 2019, encouraged B.X. to pursue research on local models of solutions to Toda systems. Our heartfelt appreciation also goes to Professor Zhaohu Nie at the University of Utah, who kindly addressed several naive questions from B.X. related to Toda systems.
Finally, we convey our profound gratitude to the anonymous reviewer whose insightful revision suggestions have significantly enhanced the manuscript's readability.

\bibliographystyle{plain}
\bibliography{model_v8}
\end{document}